\documentclass[12pt]{amsart}
\usepackage{amsthm,amsmath,amssymb}
\usepackage{addfont}
\usepackage{mathrsfs}
\usepackage[T1]{fontenc}
\numberwithin{equation}{section}

\def\ca{{\mathcal A}}

\def\cam{{\mathcal M}}
\def\cn{{\mathcal N}}

\def\he{{\mathscr E}}
\def\hf{{\mathscr F}}
\def\hg{{\mathscr G}}
\def\hh{{\mathscr H}}

\def\hk{{\mathscr K}}
\def\hl{{\mathscr L}}

\def\hn{{\mathscr N}}

\def\hr{{\mathscr R}}

\def\bc{{\mathbb C}}

\def\bn{{\mathbb N}}

\def\g{\gamma}

\def\r{\rho}
\def\s{\sigma}

\def\f{\varphi}

\theoremstyle{plain}
\newtheorem{lemma}{Lemma}[section]
\newtheorem{proposition}[lemma]{Proposition}
\newtheorem{theorem}[lemma]{Theorem}

\theoremstyle{definition}

\newtheorem{definition}[lemma]{Definition}

\theoremstyle{remark}

\begin{document}

\title[ Minimal Stinespring Representations ]{\textsc{ Minimal Stinespring Representations of  operator valued multilinear maps}}

\author[E.~Christensen]{Erik Christensen}
\address{\hskip-\parindent
Erik Christensen, Mathematics Institute, University of Copenhagen, Copenhagen, Demark.}
\email{echris@math.ku.dk}
\date{\today}

\begin{abstract}

A completely positive linear map $\f$  from a C*-algebra $\ca$  into $B(\hh)$ has a Stinespring representation as $\f(a) =  X^*\pi(a)X,$ where $\pi$  is a *-representation of $\ca$ on a Hilbert space $ \hk$  and $X$  is a  bounded operator from $\hh  $ to $\hk. $  Completely bounded multilinear operators on C*-algebras as well as some densely defined multilinear operators in Connes' non commutative geometry   also  have Stinespring representations of the form  $$ \Phi(a_1, \dots, a_k ) = X_0\pi_1(a_1)X_1 \dots \pi_k(a_k)X_k$$ such that each  $a_i$ is in a *-algebra $A_i$ and  $X_0, \dots X_k $ are densely defined closed operators between the Hilbert spaces.  We show that  for both completely bounded maps and for the geometrical maps, a natural minimality assumption implies that two such Stinespring representations have unitarily equivalent *-representations in the decomposition. 
\end{abstract}

\maketitle

\section{Introduction}
Let $\f : \ca \to B(\hh)$ be a completely bounded map from a C*-algebra $\ca $ into the algebra $B(\hh)$   of bounded operators on a Hilbert $\hh.$   The articles \cite{Ha, Pa1, Wi} by respectively Haagerup, Paulsen and Wittstock all contain the result that such a map has a Stinespring representation \cite{St}, such that there exists a Hilbert space $\hk,$ a *-representation $\pi$ of $\ca$ on $\hk$ and bounded operators $X,Y$ in $B(\hh,\hk)$ such that $$\forall a \in \ca:\quad \f(a) = X^*\pi(a)Y.$$ 

The article \cite{CS1} introduced completely bounded multilinear operators from a product $\ca_1 \times \dots \times \ca_k$ of C*-algebras into $B(\hh)$ and showed  that such  a map $\Phi$ does  have a Stinespring representation of the  form $$ \Phi(a_1,\dots, a_k) = X_0\pi_1(a_1)X_1 \dots \pi_k(a_k)X_k,$$ where the operators $X_i$ are bounded and the maps $\pi_i$ are *-represen-tations. 
In these theorems there is no explicit information describing the connection between the  possible choices for the representations $\pi_i$ and the operators $X_i.$ The theorems state that there is always an optimal choice which satisfies $\|X_0\| \|X_1\| \dots \|X_k\| = \|\Phi\|_{cb},$ the completely bounded norm of $\Phi.$ 
The present article introduces a minimality criterion - which has a long history for completely positive maps - and we show that minimal Stinespring representations do have unitarily equivalent *-representations in their decomposition. 
The mathematical description of the relation between the operators $X_i$ and $Y_i$ which will appear in 2 different minimal Stinespring representations  is given  in terms of densely defined closed unbounded operators, and with this observation in mind we realized, that it is quite natural to extend the concept of {\em bounded Stinespring representations } to the unbounded case, such that we will study a multilinear operator $\Phi$ defined on a product $ A_1 \times \dots \times A_k $ of *-algebras with values in unbounded densely defined closable operators between 2 Hilbert spaces, such that the map $\Phi$  may be given in a form analogous to the Stinespring representation of completely bounded multilinear maps,
\begin{align}  &A_1, \dots , A_k \, \text{ *-algebras,}\quad  \pi_i : A_i \to B(\hk_i) \, \,\text{*-representations,} \\ \notag & X_i \text{ closed densely defined operators, }  \\ \notag & 
 \Phi(a_1, \dots ,a_k) = \mathrm{closure}\big( X_0\pi_1(a_1)X_1 \dots \pi_k(a_k)X_k).
 \end{align}

In Connes' noncommutative geometry \cite{Co} the geometrical object \newline named {\em  a compact smooth  manifold }    is replaced  by a non commutative unital *-algebra which may play the role of  the smooth functions on the manifold. This algebra may be represented as bounded operators on a Hilbert space  and the generalization of the differentiable structure to this noncommutative setting may be studied studied via some closed unbounded densely defined operators  on the same Hilbert space. Suppose for instance that we are given a spectral triple, i.e.  a *-algebra $A$ of bounded operators on a Hilbert space $H$ and a densely defined self-adjoint operator $D$ on $H$ such that all products of commutators of the form  $[D, a_1][D, a_2]\dots [D,a_k]$ are densely defined and bounded, then the $k+1$ linear forms given as $$\Phi(a_0 , a_1, \dots, a_k) := \langle a_0[D,a_1] [D, a_2] \dots [D,a_k] \xi, \xi \rangle, \text{ some } \xi \in \hh$$ play a fundamental role both in classical differential geometry and  in the noncommutative geometry,  and they all posses  unbounded Stinespring representations as we will see.  Let $K_\xi :\hh \to \bc$ be given by $\g \to \langle \g, \xi \rangle$ and $B_\xi : \bc \to \hh $ be given by $z \to z \xi$ then 
\begin{align} \Phi(a_0, \dots, a_k) = & K_\xi  a_0 \begin{pmatrix}  D &- I \end{pmatrix} \begin{pmatrix}  a_1& 0 \\ 0 & a_1 \end{pmatrix} \begin{pmatrix}  D &- I \\ D^2 & -D\end{pmatrix} \begin{pmatrix}  a_2& 0 \\ 0 & a_2 \end{pmatrix} \dots\\ \notag & \dots \begin{pmatrix}  D &- I \\ D^2 & -D\end{pmatrix} \begin{pmatrix}  a_k& 0 \\ 0 & a_k \end{pmatrix}  \begin{pmatrix}  I \\ D\end{pmatrix} B_\xi \end{align}

 Having this example in mind, we find that it is justified to extend the study of Stinespring representations to the unbounded case too. Yet another case, where unbounded Stinespring representations appear,  is the result from \cite{C} Corollary 2.5 which shows that if $\f:\ca \to B(\hh)$ is a bounded homomorphism {\em - not *-homomorphism -} of a C*-algebra and $\hh$ is a separable Hilbert space, then there exists a *-representation $\pi:\ca \to B(\hh)$ and a bounded injective  operator $X$ with densely defined  
inverse $X^{-1}$ such that $\f(a) = X^{-1}\pi(a)X,$ i.e. $\f$ has an unbounded Stinespring representation. 

The main conclusion of the present investigation  is, that both in the case of a completely bounded multilinear operator and in the case of a multilinear operator with an unbounded Stinespring representation, there exists a minimality condition which implies that any two minimal Stinespring representations of such a map will have unitarily equivalent *-representations in their decompositions. 
The operators $X_i$ and $Y_i$ used in the two Stinespring representations  will be linked in a quite natural way. 
A similar result, with the same minimality condition for completely positive  maps is presented in Paulsen's book \cite{Pa2}, Proposition 4.2.

\section{Unbounded Stinespring representations and a lemma}     
Our motivation behind the present study was to show a certain uniqueness property for Stinespring representations of a completely bounded map, but the way the proof goes and the exact statement of the results we obtain indicate, that it is worth to introduce a study of a multilinear map from a Cartesian product of *-algebras into unbounded operators between Hilbert spaces, such that the map has  an unbounded Stinespring representation. The exact meaning of this concept will be defined now. The term {\em Stinespring representation } is used in Paulsen's book \cite{Pa2} in connection with completely bounded linear or multilinear maps on products of  C*-algebras, and when we write {\em  Stinespring representation} we always are in the completely bounded setting, whereas the words {\em unbounded Stinespring representation}  does not exclude the completely bounded case.
 With respect to the usage of words, we use the word {\em representation } to mean a *-homomorphism of an associative *-algebra $A$ over $\bc$ into the bounded operators $B(\hh)$ on some Hilbert space $\hh.$ The representation is said to be non degenerate if $I_{\hh} $ is in the strong closure of $\pi(A).$ Finally we remark that for a linear operator $X$ defined on a space $\hf,$ the restriction to a subspace $\hg $ is denoted $X|\hg.$ 
 
  To those of you, who may have forgotten some basic theory on unbounded operators we remind you that a linear  operator $T: \hh \to \hk$ between the Hilbert spaces $\hh$ and $\hk, $ which is defined on a subspace dom$(T)$ of $\hh,$ has an adjoint operator $T^*: \hk \to \hh,$ if it is densely defined. The adjoint is given by \begin{align*} \xi \in \mathrm{dom}(T^*) \iff &\,\exists \omega \in \hh  \forall \eta \in \mathrm{dom}(T): \langle T\eta, \xi \rangle  = \langle \eta, \omega \rangle \\
 \text{ in this case } T^*\xi :=& \omega.
 \end{align*} We also recall that a  densely defined operator $T$ is closable if and only if $T^*$ is densely defined, and in this case $(T^*)^* = \mathrm{closure}(T).$  
 
 The product of $ST$ of  2 operators $S, T $ is defined for any $\xi $ in dom$(T)$ such that $T\xi$ is in dom$(S)$ and no closure operation is assumed, not  even if the product is bounded on a dense domain.  
 
 For a von Neumann algebra $\cam$ on a Hilbert space $\hh$ and an unbounded densely defined closed operator $T$ on $\hh$ we say that $T$ is affiliated with $\cam$ if the operators $W$ and $|T|$ in the polar decomposition, $T = W|T|, $  \cite{KR} ,  have the property that both $W$ and $(I_\hh + |T|^2)^{-1} $ belong to $\cam.$ If $T$ is affiliated with $\cam$ then the domain of definition dom$(T)$ for $T$ is left invariant by all operators in the commutant $\cam^\prime.$ We will also use the identity dom$(T) = $ dom$(|T|).$ 
\begin{definition} 
Let $k$ be a natural number, $\ca_1, \dots \ca_k$ C*-algebras,  $\hg, \hh$   Hilbert spaces and $\Phi: \ca_1 \times \dots \times  \ca_k  \to B(\hg, \hh)$ a completely bounded $k-$linear map. Let for $1 \leq i \leq k,$ $\hk_i$ be a Hilbert space, $\pi_i: \ca_i \to B(\hk_i)$ a representation and $ X_0 \in B(\hk_1, \hh), \,  X_i \in B(\hk_{(i+1)}, \hk_i),\,  X_k \in B(\hg, \hk_k)$ such that for all $(a_1, \dots , a_k) $ in $\ca_1 \times \dots \times \ca_k :$ 
$$ \Phi(a_1, \dots , a_k )\, = \, X_0\pi_1(a_1) X_1 \dots X_{(k-1)}\pi_k( a_k)X_k.$$  The set $\{\hk_1,\dots, \hk_k, \pi_1, \dots, \pi_k, X_0, \dots  X_k\}$ is called a Stinespring representation of $\Phi.$  

Let $ 1 \leq i \leq k,$ then a Stinespring representation \newline  $\{\hk_1,\dots, \hk_k,  \pi_1, \dots, \pi_k, X_0, \dots  X_k\}$ is said to be {\em minimal } at the variable $a_i$ if 
\begin{align}  
& \mathrm{span}\big(\{\pi_i(a_i)X_i \dots \pi_k(a_k)X_k\g \, : \, a_j \in A_j, \, \g \in \hg\}\big) \text{ is dense in } \hk_i  \text{ and} \\ \notag
& \mathrm{span}\big(\{\pi_i(a_i)X_{(i-1)}^* \dots \pi_1(a_1)X_0^*\eta   \, : \, a_j \in A_j, \, \eta  \in \hh\}\big) \text{ is dense in }  \hk_i.
\end{align}
We say that the Stinespring representation is minimal if it is minimal at all the variables.   
\end{definition}  

This minimality condition is not new, it is known from the beginning of the theory of completely positive maps and it fits naturally to the setting of completely bounded maps.  It reappears in \cite{PSu} where Paulsen and Suen study another uniqueness concept for completely bounded maps. It is quite easy to see that a given Stinespring representation of a completely bounded multilinear map may always be reduced to a minimal one. 

\begin{proposition} \label{MinStCb}
Let $\Phi : \ca_1 \times \dots \times  \ca_k \to B(\hg , \hh )  $ be a completely bounded $k-$linear map with a Stinespring representation  $\{\hk_1,\dots, \hk_k,\newline \pi_1, \dots, \pi_k, X_0, \dots  X_k\}. $
There exists a tuple $(l_1, \dots, l_k)$ of orthogonal projections such that $l_i$ is in the commutant $\pi(\ca_i)^\prime$ on $ \hk_i$   and \begin{align*}\{l_1\hk_1,\dots, l_k\hk_k,  & \pi_1 |\{l_1\hk_1\}, \dots,  \pi_k|\{l_k\hk_k\},\\ & X_0|\{l_1\hk_1 \} ,\dots l_iX_i|\{l_{(i+1)}\hk_{(i+1)}\}  ,\dots , l_kX_k\}
\end{align*} is a minimal Stinespring representation for $\Phi.$ 
\end{proposition}  

\begin{proof}
We start from the right  and let $r_k$ denote the orthogonal projection in the commutant $\pi_k(\ca_k)^\prime $ with range given as the closed linear subspace of $\hk_k$ spanned by the set $\{\pi_k(a_k)X_k\g \, : \, a_k \in \ca_k,\, \g \in \hg\}.$ Then we let $l_k $ denote the orthogonal projection in $\pi_k(\ca_k)^\prime $ with range given as the closed linear span of all the vectors in the set $\{r_k\pi_k(a_k)X_{(k-1)}^*  \dots \pi_1(a_1)X_0^* \eta   \, : \, (a_1, \dots , a_k) \in \ca_1 \times \dots \times \ca_k , \, \eta \in \hh\} .$ Then we see 
\begin{align*}
\forall (a_1, \dots,  a_k) :  
\Phi(a_1, \dots , a_k) & =  X_0  \pi_1(a_1) X_1 \dots \pi_k(a_k)l_kX_k \\& \notag = \, X_0   \pi_1(a_1)X_1 \dots  X_{(k-1)}l_k \pi_k(a_k)l_k X_k. 
\end{align*}
 So if we replace $\hk_k$ by $l_k\hk_k,$ $X_{(k-1)}$  by $X_{(k-1)} l_k$ and $ X_k$ by $l_kX_k$ we have  a Stinespring representation of $\Phi$ which is minimal at the last variable place. We may then repeat the process for the variable $a_{(k-1)},$ and then obtain an orthogonal projection $l_{(k-1)} $ such that  
\begin{align} \label{k-1}
\Phi(a_1, \dots , a_k) \,& = \, X_0  \pi_1(a_1) X_1 \dots \pi_{(k-1)}(a_{(k-1)})l_{(k-1)}X_{(k-1)}\pi_k(a_k)l_kX_k . 
\end{align} We can then replace $\hk_{(k-1)}$ by $l_{(k-1)} \hk_{(k-1)}, $   $X_{(k-1)}$ by $l_{(k-1)}X_{(k-1)}l_k$ and  $X_{(k-2)}$ by $X_{(k-2)}l_{(k-1)} $ to  obtain yet another Stinespring representation of $\Phi.$ By construction it is minimal at the $(k-1)$'st variable, but we have to see that the minimality at the $k$'th variable is still in place. The linear span of the set $\{l_k\pi_k(a_k)\g \, : \, a_k \in \ca_k, \g \in \hg\} $ is left unchanged when reducing  from $\hk_{(k-1)}$ to $l_{(k-1)}\hk_{(k-1)}$ so it is dense in $l_k\hk_k.$ This density implies  that from the equation (\ref{k-1}) we can conclude that 
\begin{align} \label{k-1,1} 
&  X_0  \pi_1(a_1)X_1 \dots \pi_{(k-1)}(a_{(k-1)}) l_{(k-1)}X_{(k-1)}l_k \pi_k(a_k)  \\ \notag & = X_0  \pi_1(a_1) X_1 \dots \pi_{(k-1)}(a_{(k-1)}) X_{(k-1)}l_k \pi_k(a_k)  
\end{align} and then the closed linear span of the set $$\{\pi_k(a_k) l_kX_{(k-1)}^*l_{(k-1)}  \dots \pi_1(a_1)X_0^* \eta \, :\, a_j \in \ca_j, \eta \in \hh\}$$ equals $l_k\hk_k$ so the minimality at $a_k$ is preserved. This procedure may now be repeated and the proposition follows.

\end{proof}

As mentionened above our motivation for this work was to see how different Stinespring representations of the same completely bounded map relate to each other.
After we obtained the result which is  presented as Theorem \ref{theorem1}, we  realized that it might be extended to several variables and also that the concept we name {\em unbounded Stinespring representation } was asking for a further investigation. When we realized that this new concept has links to noncommutative geometry and to the similarity question we found it reasonable to pursue the study. The definitions we propose below are made, such that the results we want are obtainable in the setting of densely defined closed operators, but it might be that we have not found the optimal set of conditions.

\begin{definition} Let $\hg,\,  \hh $ be Hilbert spaces and $\he $ a linear subspace of $\hg.$ The space $OP(\he, \hh)$ is defined as the space of all linear operators - defined on all of $\he $ - to  $\hh.$ 
\end{definition} 

\begin{definition} 
Let $k$ be a natural number, $A_1, \dots A_k$ be associative  *-algebras over $\bc,$  $\hg, \hh$  Hilbert spaces with dense subspaces $\he $ and $\hf$ respectively and $\Phi : A_1 \times \dots \times A_k \to  OP(\he, \hh)$ a $k$-linear map. Let for $i \in \{1,, \dots k\} , \, \hk_i $ be Hilbert spaces, $\pi_i : A_i \to B(\hk_i)$ be representations, $X_0 $ a closed densely defined operator from $\hk_1 \to \hh, $ for $1 \leq i \leq k-1 $ $X_i $ a closed densely defined operator from $\hk_{(i+1)} $ to $\hk_i,$  $X_k $ a closed densely defined operator from $\hg $ to $\hk_k,$ and $\he$ a dense subspace of  $\hg$ such that for any $k-$tuple $(a _1, \dots ,a_k) $ in $A_1\times \dots \times A_k:$ 
\begin{align*}
X_0\pi_1(a_1)X_1 \dots \pi_k(a_k)X_k & \text{ is  defined and equals } \Phi(a_1, \dots, a_k) \text{ on }\he.
\end{align*}
Let $\hf$ be a dense subspace of $\hh$ we  say that the set $$( \hk_1, \dots, \hk_k, \pi_1, \dots \pi_k, X_0, \dots, X_k, \he, \hf )$$ is  an unbounded Stinespring representation of $\Phi$ if, in addition, 
\begin{equation} \forall (a_1 , \dots, a_k): \, 
 X_k^*\pi_k(a_k^*)X_{(k-1)}^*  \dots \pi_1(a_1^*)X_0^* \in OP(\hf, \hg).
 \end{equation} 
We say that the unbounded Stinespring representation is minimal if for each $i$ in $\{1, \dots ,k\}$ we have 
 \begin{align}  \label{MinCond}
 &\mathrm{span}\big(\{\pi_i(a_i)X_i \dots \pi_k(a_k)X_k\g \, : \, a_j \in A_j, \, \g \in \he\}\big) \text{ is dense in } \hk_i  \text{ and }\\
 \notag &\mathrm{span}\big(\{\pi_i(a_i)X_{(i-1)}^* \dots \pi_1(a_1)X_0^*\eta  \, : \, a_j \in A_j, \, \eta \in \hf\}\big) \text{ is dense in } \hk_i
\end{align}
\end{definition}

In the introduction above we mentioned two examples of unbounded Stinespring representations and it is known from classical - {\em commutative } - differential geometry that the classical  $(k+1)$ linear forms from differential geometry  will almost never be completely bounded. 

 The example we gave of an unbounded Stinespring representation for a bounded homomorphism of a C*-algebra into some $B(\hh)$ may or may not be completely bounded. This is - in to days language -  what Kadison's similarity question asks. 
 
 If you think that the spaces $\he$ and $\hf  $ are a bit superficial or {\em ad hoc} objects, we would like to remind you that the space of bounded smooth functions with compact support or the space of rapidly decreasing functions, known from the commutative theory, quite often appear in the roles played by $\he$ and  $\hf.$

With respect to an extension of Proposition \ref{MinStCb} to the setting of unbounded Stinespring representations, we do not know if such a result is valid in general for the multilinear case, but below we show that in the one variable case and when the operators in the image of $\Phi$ are bounded densely defined operators, the proof from the completely bounded case, just given, may be modified to work in the unbounded case too. On the other hand we think, that in the case of a multilinear operator based on a geometrical setting, there may quite often be extra information available which can be used to reduce a given  unbounded Stinespring representation  to a minimal one.

 The problems, we have faced when trying to modify the proof of Proposition \ref{MinStCb} to the unbounded case are many, and some of them  are based on the fact that, for several reasons, we insist that the operators $X_i$ which appear in an unbounded Stinespring representation have to be densely defined and closed. Unfortunately the product $bX$ of a bounded operator $b$ and an unbounded closed densely defined operator $X$ need not to be closable, and such an operation is applied several times in the proof of Proposition \ref{MinStCb}. 

\begin{proposition} \label{MinStUb} 
Let $A$ be a *-algebra, $\hg, \hh $ Hilbert spaces, $\f : A \to B(\hg, \hh)$   a linear map with an unbounded Stinespring representation $(\hk, \pi , X_0, X_1, \he, \hf).$ Then $\f$ has a minimal unbounded Stinespring representation.
\end{proposition}

\begin{proof}
The proof follows that of Proposition \ref{MinStCb}, so let $r$ be the orthogonal projection in the commutant $\pi(A)^\prime$ which maps onto the closed linear span of the set $\{\pi(a)X_1\g\, :\, a \in A, \g \in \he\}, $ and let $l $ denote the orthogonal projection in $\pi(A)^\prime$ which maps onto the closed linear span of the set $\{\pi(a)rX^*_0\eta\, : \, a \in A, \eta \in \hh\}.$ Then for any $a$ in $A,$ any  $\g$ in $\hg$ and any $\eta \in \hh$ we have \begin{equation} \label{CloX}   \langle \pi(a)X_1\g, X^*_0\eta \rangle \, = \, \langle  \pi(a)l X_1\g, lX^*_0\eta \rangle, \end{equation} and the set $\big( l\hk, \pi|l\hk, (lX^*_0|\hf)^*, \mathrm{closure}(lX_1|\he)\big)$ is a candidate for a minimal unbounded Stinespring representation, except that we have to show that $(lX^*_0 )^*$ is densely defined and that $lX_1$ is closable. Before we show that the candidates do work, we remark that the minimality conditions to the left and right are fulfilled by the construction. 
We start by showing that $lX_1|\he$ is closable, so we let $(\g_n)$ denote a sequence from $\he$ and $\xi$ a vector in $\hk$ such that the sequence $(\g_n,lX_1 \g_n)$ converges in norm to $(0,\xi).$ We will then have to show that $\xi = 0,$ in order to show the closability.    By (\ref{CloX}) we have 
\begin{align*} \forall a \in A, \, \forall \eta \in \hf
\langle \xi, l\pi(a)X^*_0\eta \rangle &= \lim \langle lX_1\g_n, l\pi(a)X^*_0\eta \rangle \\
&= \lim \langle \f(a) \g_n, \eta \rangle \\&= 0.  
\end{align*}

Then $\xi $ in $l\hk$ is orthogonal to the vectors in a dense subspace of $l\hk$ so  $\xi = 0 $ and $lX_1|\he$ is closable. We will then  define  $\hat X_1 := \mathrm{closure}(lX_1|\he).$ 
The operator $l X^*_0|\hf$ may be shown to be closable in the same way as we just showed that $lX_1 |\he$ is closable. Since $lX^*_0|\hf$ is densely defined, it has an adjoint operator $(lX_0^*|\hf)^*$ and this operator is densely defined since $lX_0^*|\hf$ is closable.
 We can then define $\hat X_0$ as the closed densely defined operator $ \hat
X_0 := (X_0^*|\hf)^*.$ We have a candidate for the minimal unbounded Stinespring representation, but we still need to see that $\f(a)|\he = \hat X_0 \pi(a)l \hat X_1|\he,$ so we start computing.
   
\begin{align*} \forall a \in A, \, \forall \g \in \he, \,  \forall \eta \in \hf: \, 
\langle \pi(a)lX_1 \g, lX^*_0\eta \rangle  &= \langle \f(a) \g, \eta \rangle .\\   
\end{align*}
This means that $\pi(a)lX_1\g $ is in dom$\big((lX^*_0|\he)^*\big)$ and $$(lX_0^*|\he)^*\pi(a)lX_1 \g   = \f(a)\g,$$  so $\hat X_0 \pi(a) l \hat X_1 |\he \, = \, \f(a)|\he $ and the set $(l\hk, \pi| l\hk, \hat X_0, \hat X_1, \he, \hf)$ is a minimal unbounded Stinespring representation for $\f,$ as claimed. 
\end{proof}
In \cite{PSu} Paulsen and Suen study Stinespring representations of a completely bounded map $\f$ of the form $\{K, \pi, V^*, TV\},$ such that  $T$ is a bounded operator on $K$ in the commutant $\pi(\ca)^\prime.$ They show that  two minimal  Stinespring representations  of this sort  do have unitarily equivalent *-representations in their Stinespring representations.  

We will extend this result to unbounded Stinespring representations of multilinear maps.  The multilinear case is proved via an induction argument. The induction step and the one variable case is to a large extent proven as an application of the lemma we present below.

\begin{lemma} \label{lemma}
Let $A$ be an associative *-algebra over $\bc,$ $\hg, \hk, \hl$
 Hil-bert spaces,  $\pi: A \to B(\hk),$ $\r: A \to B(\hl)$ non degenerate *-represen-tations of $A$ and $\s := \pi \oplus \r$ the direct sum representation of $A$ on $\hk \oplus \hl.$ Let $X : \hg\to \hk $ and $Y: \hg \to \hl$ be  densely defined closed operators, which both have a dense subspace $\he$ of $\hg $ in their domain of definition. Let $N$ be a projection in the commutant $\s(A)^\prime$ such that all  the vectors $\{(\pi(a) X\g, \r(a)Y \g)\, : \, a \in A, \g \in \he\}$ are in the space $\hn:= N(\hk \oplus \hl).$    Let $K$ denote the orthogonal projection from $\hk \oplus \hl  $ onto $\hk.$ 
If \begin{equation} \label{generic} N\wedge K = N\wedge (I-K) = (I-N)\wedge  K  = (I-N) \wedge (I-K) = 0 \end{equation}  then there exists a closed densely defined operator $T : \hk  \to \hl, $  such that for the polar decomposition $T = W|T| :$
\begin{itemize}
\item[(i)] $\hn = N(\hk \oplus \hl) = G(T),$ the graph of $T,$ \item[(ii)] $\forall \g \in \mathrm{dom}(T)\, \forall a \in A: \, \pi(a)\g \in \mathrm{dom}(T)$ and $T\pi(a)\g = \r(a)T\g,$ 
\item[(iii)] the isometry $W,$  in the polar decomposition $T = W|T|,$ maps $\hk$ onto $\hh$  such that $ \forall a \in A: \, \, W\pi(a)W^* = \r(a),$ and the positive operator $|T| $ is affiliated with the commutant $\pi(A)^\prime,$ 
\item[(iv)] $ \forall \g \in \he: \, X\g \in \mathrm{dom}(T) \text{ and } TX\g = Y\g.$ 
\end{itemize} 
  \end{lemma}

\begin{proof}
We will work with operators on the direct sum Hilbert space $\hk \oplus \hl$ and write them as matrices, when this notation is the easiest to use. The direct sum representation $\s := \pi \oplus \r $ may then be written in the form $$ \s(a) = \begin{pmatrix}
\pi(a) & 0 \\ 0 &\r(a)
\end{pmatrix}.$$ 

We will apply Halmos' result \cite{Hal} describing the relations between two closed subspaces of a Hilbert space  to the 2 closed subspaces $\hk$ and $\hn$ of $\hk  \oplus \hl.$  By assumption the two subspaces are in the so-called generic position to each  other and  by Halmos' result \cite{Hal} there exists a closed densely defined operator $T$ from $\hk$ to $\hl$ such that $\hn = G(T),$ the graph of $T,$ and $T$ is injective with dense range. 
The fact that $N$ is in the commutant of $\s(A)$  means that the closed subspace $\hn = G(T)$ is invariant for all  the operators $\s(a),$ so \begin{equation} \label{Ttwin}
\forall a \in A\, \forall \xi \in \text{dom}(T):\quad  \pi(a)\xi \in \text{dom}(T), \text{ and } T\pi(a)\xi = \r(a) T \xi.
\end{equation}

It follows that for the polar decomposition $T = W|T|,$ which in the unbounded case is presented in  \cite{KR} 6.1.11,  the positive operator  $|T|$ on $\hk$ is affiliated with the commutant $\pi(A)^\prime$ and $W$ is an isometry of $\hk$ onto $\hl$ such that
\begin{equation} \label{WunEq}
\forall a \in A; \quad \r(a) = W\pi(a)W^*.
\end{equation}
These consequences may be seen from the fact that the positive operator, say $N_{11},$ being the upper left hand element in the matrix for the projection $N$ is  in $B(\hk)$  and  commutes with all the operators in $\pi(A)$. It follows from \cite{Hal} that  $N_{11} = (I_\hk + T^*T)^{-1},$ so we get that $|T| := (T^*T)^{\frac{1}{2}}$is affiliated to $\pi(\ca)^\prime,$ and the domain of definition for $T$ is left invariant by all the operators $\pi(a).$  In the polar decomposition $W|T| := T, $  the isometry part $W$ must map $\hk$ isommetrically onto $\hl$ since $T$ is injective and the range of $T$ is dense. Hence 
\begin{align} \notag
\forall a \in A: \quad \r(a)W|T| & =W| T| \pi(a) = W\pi(a)|T| \\ \notag \forall a \in A: \quad W^* \r(a)W & =\pi(a).
\end{align}
By assumption we know that for $a$ in $A$ and $\g $ in $\he$  we have $\pi(a)X\g$ in dom$(T)$ and $T\pi(a)X \g = Y\r(a)\g,$ and the statement (iv) of the lemma follows when $A$ is unital. In the non unital case we know by Kaplansky's density theorem and the assumption on non degenerate representations $\pi, \r$  that there exists a net of contractions $(\pi(a_i) )_{(i \in J)}$ in $\pi(A)$ which converges strongly to $I_\hk.$ Since $\r(a) = W\pi(a)W^*$ the net $( \r(a_i))_{(i \in J)}$ converges strongly to $I_\hl$ and we get for any $\g $ in $\he$ that 
$$  \lim(\pi(a_i)X\g, T\pi(a_i)X\g) = \lim ( \pi(a_i)X\g, \r(a_i) Y \g) = (X\g, Y \g,)$$
since $T$ is closed it follows that $X\g $ is in dom$(T)$ and $TX\g = Y \g.$ 
The lemma follows.  
\end{proof}

  \section{Equivalence of unbounded Stinespring representations } 
We will now work with  the definitions and results of the previous section  to show that two minimal unbounded Stinespring representations of a multilinear map have unitarily equivalent *-representations in their decompositions. The proof is a finite induction argument and the induction step consists to a large extent in an application of Lemma \ref{lemma}. This general theorem may then be applied to the case of two minimal Stinespring representations of one completely bounded map. On the other hand the general unbounded case is filled with small arguments concerning domains of definition for the unbounded operators in the proof. We have chosen to present the completely bounded one variable case first, to make the proof of the general unbounded multilinear case easier to follow. It also turns out that the completely bounded one variable  case has a proof which is nearly identical to the proof of the induction step in the general unbounded multilinear case, so we will save a bit in the proof of the general theorem.    

 \begin{theorem} \label{theorem1}
 Let $\ca$ be a C*-algebra,  $\hg, \hh, \hk , \hl,$   Hilbert spaces and   $ \f: \ca \to  B(\hg, \hh)$  a completely bounded map. Let 
$
\{\hk, \pi, X_0, X_1 \}$  and $ \{\hl, \r, Y_0, Y_1\} $ 
be two minimal Stinespring representations of $\f,$ then there exists  a densely defined, closed, injective operator operator $T : \hk\to \hl$ with dense range, such that for the polar decomposition $T =  W|T|$ 
\begin{itemize} \item[(i)] $\forall \xi \in \mathrm{dom}(T)\, \forall a \in \ca:\, \pi(a)\xi \in \mathrm{dom}(T)$ and $T\pi(a)\xi = \r(a)T\xi,$   \item[(ii)] $W$ is an isomtery of $\hk$ onto $\hl,$ and $\forall a \in \ca: \, W\pi(a)W^* = \r(a),$ 
\item[(iii)] $\forall \g \in \hg: X_1\g  \in \mathrm{dom}(T) : TX_1\g  = Y_1\g,$ 
\item[(iv)] 

$\forall \eta \in \hh : Y_0^*\eta \in \mathrm{dom}(T^*) \text{ and } T^*Y_0^* =X^*_0,$ 

$  Y_0T = X_0| \mathrm{dom}(T).$
\end{itemize} 
\end{theorem}

\begin{proof}
Define $\hr$ as the closed linear span inside $\hk \oplus \hl$ of the vectors in the set $\{(\pi(a)X_1 \g, \r(a)Y_1\g)\, : \, a \in \ca, \, \g \in \hg\},$ and let $\hn $ denote the closed linear subspace  of $\hk \oplus \hl$ defined as an  intersection of perpendicular complements 
$$\hn := \underset{ a \in \ca,\, \eta \in \hh }{\cap} (- \pi(a) X_0^*\eta, \r(a)Y^*_0\eta)^\perp.$$ Since we look at two Stinespring representations of $\f,$
it turns out that $\hr$ is a subspace of $\hn,$ then when we  define $R$ and $N$ as the orthogonal projections onto respectively $\hr$ and $\hn$ we get $N \geq R.$ As in the lemma above we let $\s$ denote the representation of $\ca $ on $ \hk \oplus \hl$ given as $ \s := \pi \oplus \r.$ By construction both $\hn$ and $\hr$ are invariant under $\s(\ca),$ so both $R$ and $N$ belong to the commutant $\s(\ca)^\prime.$  We will now apply  Lemma \ref{lemma} so we define $K, \, L$ as the orthogonal projections from $ \hk \oplus \hl$ onto $\hk$ and $\hl$ respectively. 

The minimality condition, span$(\{\pi(a)X_0^*\eta\, :\, a \in \ca, \, \eta \in \hh\}) $ is dense in $\hk,$ implies that $K\wedge N =0.$  

The minimality condition, span$(\{\r(a)Y_0^*\eta \, :\, a \in \ca, \, \eta \in \hh\}) $ is dense in $\hl,$ implies that $(I-K) \wedge N =  L\wedge N =0.$  

The minimality condition, span$(\{\pi(a)X_1\g \, :\, a \in \ca, \, \g  \in \hg\}) $ is dense in $\hk,$ implies that $K\wedge(I- R) =0,$ and since $I-N \leq I-R$ we have $K\wedge (I-N) =0.$  

The minimality condition, span$(\{\r(a)Y_1\g \, :\, a \in \ca, \, \g  \in \hg\}) $ is dense in $\hl,$ implies that $L\wedge(I- R) =0,$ and since $I-N \leq I-R$ we have $(I-K) \wedge (I-N) = L\wedge (I-N) =0.$  

Then $N$ and $K$ are in the generic position,described in the proof of Lemma \ref{lemma}, so that lemma  implies that there exists a densely defined closed operator $T: \hk \to \hl$ such that the conditions (i), (ii) and (iii) of this theorem are fulfilled. 

With respect to the statement (iv), we remark that by the  definition of $\hn$ and $T$  we have $$ \forall \xi \in \mathrm{dom}(T) \, \forall a \in \ca\, \forall \eta \in \hh: \,\,  \langle\xi, \pi(a)X_0^* \eta \rangle \, = \, \langle T\xi , \r(a)Y_0^*\eta \rangle,$$
hence 
$$ \forall a \in \ca\, \forall \eta \in \hh:  \,\,  \r(a) Y^*_0\eta  \in \mathrm{dom}(T^*) \text{ and } T^*\r(a)Y^*_0\eta = \pi(a)X^*_0\eta$$ Just as in the proof of (iv) in Lemma \ref{lemma} we remark that $T^*$ is a closed operator and Kaplansky's density theorem may then be applied again to show $$  \forall \eta \in \hh:  \,\,   Y^*_0\eta  \in \mathrm{dom}(T^*) \text{ and } T^*Y^*_0\eta = X^*_0\eta,$$
then $T^*Y^*_0  = X_0^*$ and since in general for closed densely defined operators $A, B $ with $AB$ densely defined  we have $B^*A^* \subseteq (AB)^*$ we get $Y_0T = X_0 | \mathrm{dom}(T), $ 
and the theorem follows.
\end{proof}

We will now turn  to the general case where we consider a multilinear map with two {\em  minimal  unbounded } Stinespring representations. 

 \begin{theorem} \label{theorem2}
 Let $k$ in $\bn$, $ A_1, \dots A_k $ *-algebras,  $\hg, \hh, $   Hilbert spaces, $\he$ a dense linear subspace of $\hg,$ $\hf$ a dense linear subspace of $\hh$ and   $ \Phi : A_1\times \dots \times A_k \to OP(\he,\hh)$ a $k-$linear map. Let 
\begin{align*}
&\{\hk_1, \dots, \hk_k, \pi_1, \dots, \pi_k , X_0, \dots, X_k, \he, \hf\} \text{ and } \\ &
\{\hl_1, \dots, \hl_k, \r_1, \dots, \r_k , Y_0, \dots, Y_k, \he, \hf \} \end{align*}
be two minimal unbounded  Stinespring representations of $\Phi.$ Define for $1\leq i\leq k$ subspaces $\he_i $ of $\mathrm{dom}(X_{(i-1)}) \subseteq \hk_i$ as  the linear span of all vectors of the form 
$$ \pi_i(a_i)X_i\dots \pi_k(a_k)X_k\g \, \, \text{ such that  } a_j \in A_j, \g \in \he.$$   There exists for each $i $ in $\{1,  \dots ,k\} $ a densely defined, closed, injective  operator $T_i : \hk_i \to \hl_i$ with dense range such that for the polar decomposition $T_i =  W_i|T_i|$ 
\begin{itemize} \item[(i)] $\forall \xi \in \mathrm{dom}(T_i)\, \forall a_i \in A_i:\, \pi_i(a_i)\xi \in \mathrm{dom}(T_i)$ and $ \newline T_i\pi_i(a_i)\xi = \r_i(a_i)T_i\xi,$   
\item[(ii)] $W_i$ is an isometry of $\hk_i$ onto $\hl_i,$ and  $ \newline \forall a_i \in A_i: \, W_i\pi_i(a_i)W_i^* = \r_i(a_i),$ 
\item[(iii)] $\forall \g \in \he: X_k\g  \in \mathrm{dom}(T_k) : T_kX_k\g  = Y_k\g,$ 
\item[(iv)] $ \forall i \in \{1, \dots, k-1\} : \newline  \forall \xi \in \he_{(i+1)}: \xi \in \mathrm{dom}(T_{(i+1)}) \text{ and }  T_{(i+1)}\xi \in \mathrm{dom}(Y_i) , \newline   X_i\xi \in  \mathrm{dom}(T_i),  \text{ and }   T_i  X_i\xi = Y_iT_{(i+1)}\xi $
\item[(v)] $\forall \xi \in \he_1:\, T_1\xi \in \mathrm{dom}(Y_0)\text{ and } X_0\xi =   Y_0T_1\xi.$ $\newline \forall \eta \in \hf :\,       Y_0^*\eta \in \mathrm{dom}(T_1^*) \text{ and } T_1^*Y_0^* \eta = X_0^* \eta.$   
\end{itemize} 
\end{theorem}
 \begin{proof}
 The proof is a kind of a finite induction argument where we reduce the number of variables by one in each step until we end up proving item (v) in the theorem. On the other hand the operators $T_i$ which appear in the result only  depend on the 2 given unbounded Stinespring representations, and not on  the process of reducing  the number of variables,  so we will construct the $T_i\,'$s before the reduction in the number of variables begin. We proceed as in the proof of Theorem \ref{theorem1},  so $\hr_i $ is defined as the closed linear subspace of $\hk_i \oplus \hl_i$ spanned by all the vectors of the form $(\pi_i(a_i)X_i \dots \pi_k(a_k)X_k \g , \r(a_i)Y_i \dots \r_k(a_k)Y_k\g )$ with $\g$ in $\he$ and  $(a_i , \dots, a_k) $ in $A_i \times \dots \times A_k.$  Then $\cn_i$ is defined as the closed linear subspace of $\hk_i \oplus \hl_i$ which is the intersection of all the  perpendicular complements 
 $$ (- \pi_i(a_i) X_{(i-1)}^* \dots \pi_1(a_1)X_0^* \eta, \r(a_i) Y_{(i-1)}^* \dots \r_1(a_1)Y_0^*\eta )^\perp,$$
 with $ (a_1, \dots, a_i) \in A_1 \times \dots \times A_i$ and $ \eta \in \hf.$ 
 As before we define orthogonal projections $R_i \leq N_i$ in the commutant $\s(A_i)^\prime$ as the orthogonal projections with ranges $\hr_i $ and $\hn_i,$ respectively. The minimality conditions on the two unbounded Stinespring representations do - as in the proof of Theorem \ref{theorem1} - imply that the projections $K_i$ and $N_i$ are in the generic position, so for each $i$ there exists a  closed densely defined injective operator $T_i$ with dense range   from $\hk_i$ to $\hl_i$ such that the items (i) and (ii) of this theorem are established.
 
The relation $\hr_i \subseteq \hn_i = G(T_i)  $ implies the following equation
\begin{align}  \notag
&\forall i \in \{1, \dots, k\} \forall \g \in \he,\, \forall (a_i, \dots, a _k) \in A_1 \times \dots \times A_k: \\ \notag &\pi_i(a_i) X_i \dots \pi_k(a_k)X_k \g \in \mathrm{dom}(T_i) \text{ and } \\ \label{TiRel}  &
T_i \pi_i(a_i) X_i \dots \pi_k(a_k)X_k \g \, = \, \r_i(a_i) Y_i \dots \r_k(a_k)Y_k \g, \\ & \notag \he_i \subseteq \mathrm{dom}(T_i).
\end{align}

 For $i =k $ the equation (\ref{TiRel}) becomes 
 \begin{equation} \label{iiia} \forall \g \in \he\, \forall a_k \in A_k: \, \pi_k(a_k)X_k \g \in \mathrm{dom}(T_k
) , \,\,  T_k\pi_k(a_k)X_k \g = \r_k(a_k)Y_k\g.\end{equation}   Again, as in the proof of (iii) in Theorem \ref{theorem1},  the item (ii) of this theorem, the closedness of $T_k$  and Kaplansky's density theorem show, when applied to (\ref{iiia}), that item (iii) is valid. 

We can now start the reduction in the number of variables and we reduce from the right. By the proven items (i), (ii) and (iii)  we know that 
\begin{align}
\notag \forall \g \in \he \, \forall (a_1 , \dots , a_k ) &\in A_1 \times \dots A_k : \\ \notag X_0\pi_1(a_1)X_1 \dots \pi_k(a_k)X_k \g &=  Y_0\r_1(a_1)Y_1 \dots Y_{(k-1)}\r_k(a_k) Y_k \g \\ \notag \text{ by (iii) } & =   Y_0\r_1(a_1)Y_1 \dots Y_{(k-1)} \r_k(a_k)T_kX_k\g \\ 
\label{Yk-1Tk}  \text{ by (i) } & =   Y_0\r_1(a_1)Y_1 \dots Y_{(k-1)}T_k  \pi_k(a_k)X_k\g \\ \notag  \text{ so } &\\ \label{reduc}\forall \xi  \in \he_k \,\, \forall (a_1 , \dots , a_{(k-1)} ) &\in A_1 \times \dots A_{(k-1)}  :  \\  \notag  
X_0 \pi(a_1)\dots \pi_{(k-1)}(a_{(k-1)})X_{(k-1)}  \xi &=  Y_0\r(a_1)\dots  \r_{(k-1)}(a_{(k-1)}) Y_{(k-1)}T_k \xi.  
\end{align}

This equation shows that we may define a $(k-1)$ linear map, say $\Phi_{(k-1)}$ of $A_1 \times \dots \times A_{(k-1)} $ into $OP(\he_k, \hh)$ which is given as 
\begin{align}  \notag
\forall \xi \in \he_k: \, \Phi_{(k-1)} (a_1 ,\dots , a_{(k-1)})\xi  \, :&=  
X_0 \pi(a_1)\dots \pi_{(k-1)}(a_{(k-1)})X_{(k-1)}  \xi \\ \label{Fik-1} &=  Y_0\r(a_1)\dots  \r_{(k-1)}(a_{(k-1)}) Y_{(k-1)}T_k \xi.  
\end{align}

The equation {\em almost }  shows that $\Phi_{(k-1)}$ has two unbounded Stinespring representations, the only thing we are missing is to show that $Y_{k-1)}T_k|\he_k$ is closable. On the other hand for $\xi $ in $\he_k,$  $\eta $ in $\hf$ and $(a_1 , \dots, a_{(k-1)}) $ in $A_1 \times \dots \times A_{(k-1)} $
\begin{align*}
&\langle Y_{(k-1)}T_k \xi , \r_{(k-1)}(a_{(k-1)}^*)Y_{(k-2)}^* \dots Y_0^*\eta \rangle\\ &= \langle \Phi_{(k-1)}( a_1, \dots , a_{(k-1)}) \xi, \eta \rangle\\ 
&= \langle \xi, X_{(k-1)}^* \pi_{(k-1)}(a_{(k-1)}^*) \dots X_0^* \eta \rangle. 
\end{align*}
 This shows that the densely defined operator $(Y_{(k-1)}T_k| \he_k)$ has an  adjoint operator which is defined on the dense subspace \newline
 span($\{  \r_{(k-1)}(a_{(k-1)}^*)Y_{(k-2)}^* \dots Y_0^*\eta \, : \, (a_1, \dots , a_{(k-1)} ) \in A_1 \times A_{(k-1)}, \eta \in \hf \})$ with 
 \begin{align} \label{Yk-1Tk*} 
 &(Y_{(k-1)}T_k|\he_k)^* \r_{(k-1)}(a_{(k-1)}^*)Y_{(k-2)}^* \dots Y_0^*\eta \\ \notag& =X_{(k-1)}^* \pi_{(k-1)}(a_{(k-1)}^*) \dots X_0^* \eta.
\end{align}   We have then obtained 2 unbounded Stinespring representations for $\Phi_{(k-1)} ,$ 

\begin{align*}
&\{\hk_1, \dots, \hk_{(k-1)}, \pi_1, \dots, \pi_{(k-1)} , X_0, \dots, X_{(k-1)}, \he_k , \hf\} \text{ and } \\  &
\{\hl_1, \dots, \hl_{(k-1)}, \r_1, \dots, \r_{(k-1)} , Y_0, \dots, (Y_{(k-1)}T_k|\he_k)^{**}, \he_k , \hf \}, 
\end{align*}
In order to show that these unbounded Stinespring representations are minimal we have to show that the minimality conditions in (\ref{MinCond}) are satisfied. The conditions to the left i.e. involving $X_{(i-1)}^*$ and $Y_{(i-1)}^* $ are the same as the ones for the given two minimal Stinespring representations of $\Phi$ except that $i \leq (k-1)$ now.  

With respect to the conditions  to the right we will discuss the ones involving the $X_j$'s first. So let $i = k-1$ then by the definition of $\he_{(k-1)}$ 
\begin{equation} \label{MinkL}
\mathrm{span}(\{\pi_{(k-1)}(a_{(k-1)})X_{(k-1)} \xi\, : \, a_{(k-1)} \in A_{(k-1)}, \, \xi \in \he_k \} \, = \, \he_{(k-1)} .
\end{equation}
Hence the conditions to the left for the $X_i$ representation of $\Phi_{(k-1)} $ follows from the assumptions on the $X_i$ representation of $\Phi.$ 

The investigation of minimality conditions to the left for  the $Y_i$ representation of $\Phi_{(k-1)}$ starts with the following equation
\begin{align} \label{MinkR} 
\mathrm{span}(\{\r_{(k-1)}&(a_{(k-1)})Y_{(k-1)}T_k\xi\, : \\&\notag a_{(k-1)} \in A_{(k-1)}, \, \xi \in \he_k \} \, \\  = \notag \, \mathrm{span}(\{\r_{(k-1)}&(a_{(k-1)})Y_{(k-1)}T_k\pi_k(a_k)X_k\g \, : \\& \notag a_{(k-1)} \in A_{(k-1)}, a_k \in A_k, \g \in \he \}  
\\ \notag &\text{ by (i) } \\ \notag  =   \mathrm{span}(\{\r_{(k-1)}&(a_{(k-1)})Y_{(k-1)}\r_k(a_k)T_kX_k \g \, : \\ & \notag a_{(k-1)} \in A_{(k-1)}, a_k \in A_k, \g \in \he \}
\\ \notag 
&\text{ by (iii) } \\ \notag  =   \mathrm{span}(\{\r_{(k-1)}&(a_{(k-1)})Y_{(k-1)}\r_k(a_k)Y_k \g \, : \\ & \notag \, a_{(k-1)} \in A_{(k-1)}, a_k \in A_k, \g \in \he \}.
\end{align}
The last span is exactly the one used in the last  minimality  conditions for  the representation of $\Phi$  based on the $Y_i$'s and the following ones for $\Phi$ are then exactly the ones needed to show that the $Y_i$ representation for $\Phi_{(k-1)}$ is minimal. 

We will now establish the validity of the first line of item (iv) in the case $i = k-1.$ The equation (\ref{TiRel}) shows that $\he_k \subseteq \mathrm{dom}(T_k),$ and the computations leading to equation (\ref{Yk-1Tk}) show that for $\xi$ in $\he_k$ we have $T_k\xi$ in dom$(Y_{(k-1)}).$  

To prove the second line of item (iv) in the case $i = (k-1)$ we first remark that the identities established in (\ref{MinkL}) and (\ref{MinkR}) imply that the relations defining $\hr_{(k-1)} $ and $\hn_{(k-1)}$ will be unchanged when we move from $ \Phi $ to $\Phi_{(k-1)}$  so the operator $T_{(k-1)} $ will be the same in both the $\Phi$ and the $\Phi_{(k-1)}$ case, and from the  
statement (iii) in the theorem, but applied to the $\Phi_{(k-1)} $  case we get 
\begin{align} & \notag \forall \xi \in \he_k\, : \, X_{(k-1)} \xi \in \mathrm{dom}(T_{(k-1)}) \text{ and } T_{(k-1)} X_{(k-1)} \xi = Y_{(k-1)} T_k\xi,
\end{align} 
which is the second line in the item (iv) for $i = (k-1).$  

We can then replace $k$ by $(k-1)$ and $\he$ by $\he_k,$ and so on and continue to reduce from $(k-1)$ to $(k-2)$ variables until we have obtained the one variable case. 

\begin{align} \notag
\forall a_1 \in A _1 \forall \xi  \in \he_2: \xi \in \mathrm{dom}(T_2), \,  X_1\xi &\in \mathrm{dom} (T_1),\,  T_2\xi \in \mathrm{dom}(Y_1) \text{ and }
\\ \text{ by (iii) for } \Phi_2 \quad 
T_1X_1\xi \, &= \, Y_1T_2\xi \label{TX1} \\ \notag  \text{ by } (\ref{Fik-1}) \text{ for } \Phi_1\quad  X_0  \pi_1(a_1)X_1\xi \,& = \,Y_0\r_1(a_1)Y_1T_2\xi\\ \notag \text{ by (\ref{TX1}) } \, &= \, Y_0\r_1(a_1)T_1X_1\xi
\\ \notag \text{ by item (i) } \, &=\, Y_0T_1\pi_1(a_1)X_1\xi,\end{align} and then \begin{equation} \forall \xi \in \he_1: \, X_0\xi \, = \, Y_0T_1\xi.
\end{equation} 
The first part of item (v) is established, and for the second we remark that by the definition of $\hn_1$ we have the following identity 
\begin{align*}
\forall \xi \in \mathrm{dom}(T_1)\, \forall \eta \in \hf\,  \forall a_1 \in A_1 & : \\
\langle \xi, \, \pi_1(a_1)X^*_0 \eta \rangle \,&=\, \langle T_1 \xi , \, \r_1(a_1)Y_0^* \eta \rangle.
\end{align*}
Then $\r_1(a_1)Y_0^* \eta $ is in dom$(T_1^*)$ and $T_1^*\r_1(a_1)Y_0^*\eta = \pi_1(a_1)X_0^*\eta.$ Since $T_1^*$ is a closed operator we may repeat the argument given in the proof of item (iii) of Theorem \ref{theorem1}  to extend the equality just proven to get $$ \forall \eta 
\in \hf : Y^*_0\eta \in \mathrm{dom}(T_1^*) \text{ and } T_1^* Y^*_0\eta = X^*_0 \eta,$$ and the theorem follows.

 \end{proof}

\end{document}